\documentclass[12pt,letterpaper]{article}
\pdfoutput = 1

\usepackage{graphics,epsfig,graphicx,color}
\graphicspath{{/EPSF/}{../Figures/}{Figures/}}

\usepackage{amssymb,latexsym}

\usepackage{setspace}

\usepackage{amsmath}

\usepackage{mathrsfs,amsfonts}

\usepackage{amsthm,amsxtra}

\newif\ifPDF
\ifx\pdfoutput\undefined
	\PDFfalse
\else
	\ifnum\pdfoutput > 0
		\PDFtrue
	\else
		\PDFfalse
	\fi
\fi

\ifPDF
	\usepackage{pdftricks}
	\begin{psinputs}
		\usepackage{pstricks}
		\usepackage{pstcol}
		\usepackage{pst-plot}
		\usepackage{pst-tree}
		\usepackage{pst-eps}
		\usepackage{multido}
		\usepackage{pst-node}
		\usepackage{pst-eps}
	\end{psinputs}
\else
	\usepackage{pstricks}
\fi

\ifPDF
	\usepackage[debug,pdftex,colorlinks=true, 
	linkcolor=blue, bookmarksopen=false,
	plainpages=false,pdfpagelabels]{hyperref}
\else
	\usepackage[dvips]{hyperref}
\fi

\usepackage{fancyhdr}

\DeclareMathOperator{\sech}{sech}

\newtheorem{theorem}{Theorem}[section]
\newtheorem{lemma}[theorem]{Lemma}

\newcommand{\aver}[1]{\langle {#1} \rangle}

\newcommand{\bLambda}{\boldsymbol \Lambda}

\newcommand{\bbR}{\mathbb R}

 \newcommand{\bff}{\mathbf f}
\newcommand{\bg}{\mathbf g}

 \newcommand{\bn}{\mathbf n}
 \newcommand{\bp}{\mathbf p}

 \newcommand{\bx}{\mathbf x} 
\newcommand{\by}{\mathbf y} \newcommand{\bz}{\mathbf z}

\newcommand{\bI}{\mathbf I}

\newcommand{\cA}{\mathcal A} 
\newcommand{\cC}{\mathcal C}  
 
 \newcommand{\cH}{\mathcal H}

\DeclareMathOperator*{\argmin}{arg\,min}

\newenvironment{keywords}
{\noindent{\bf Key words.}\small}{\par\vspace{1ex}}

\title{A one-step reconstruction algorithm for quantitative photoacoustic imaging}

\author{
	Tian Ding\thanks{Department of Mathematics, University of Texas, Austin, TX 78712; tding@math.utexas.edu}
\and
	Kui Ren\thanks{Department of Mathematics \& ICES, University of Texas, Austin, TX 78712; ren@math.utexas.edu}
\and
	Sarah Vall\'elian\thanks{Department of Mathematics, University of Texas, Austin, TX 78712; svallelian@math.utexas.edu}
}

\begin{document}

\date{}

\maketitle



\begin{abstract}
Quantitative photoacoustic tomography (QPAT) is a recent hybrid imaging modality that couples optical tomography with ultrasound imaging to achieve high resolution imaging of optical properties of scattering media. Image reconstruction in QPAT is usually a two-step process. In the first step, the initial pressure field inside the medium, generated by the photoacoustic effect, is reconstructed using measured acoustic data. In the second step, this initial ultrasound pressure field datum is used to reconstruct optical properties of the medium. We propose in this work a one-step inversion algorithm for image reconstruction in QPAT that reconstructs the optical absorption coefficient directly from measured acoustic data. The algorithm can be used to recover simultaneously the absorption coefficient and the ultrasound speed of the medium from \emph{multiple} acoustic data sets, with appropriate \emph{a priori} bounds on the unknowns. We demonstrate, through numerical simulations based on synthetic data, the feasibility of the proposed reconstruction method.
\end{abstract}






\section{Introduction}
\label{SEC:Intro}

Photoacoustic tomography (PAT) is a recent multi-physics biomedical imaging modality that aims at achieving simultaneously high resolution and high contrast in imaging by coupling the high-resolution ultrasound imaging modality with the high-contrast diffuse optical tomography (DOT) modality. In PAT, near infra-red (NIR) photons are sent into an optically absorbing and scattering medium, for instance a piece of biological tissue, $\Omega \subseteq\bbR^d$ ($d\ge 2$), where they diffuse. The density of the photons, denoted by $u(\bx)$, solves the following diffusion equation~\cite{ArSc-IP09,BaUh-IP10}
\begin{equation}\label{EQ:Diff}
	\begin{array}{rcll}
		-\nabla\cdot D(\bx) \nabla u(\bx) + \sigma(\bx) u (\bx) &=& 0, & \mbox{in}\ \ \Omega \\
		u &=& g(\bx), & \mbox{on}\ \ \partial \Omega
	\end{array}
\end{equation}where $D$ and $\sigma$ are respectively the diffusion and absorption coefficients of the medium, and $g$ is the incoming photon source. The medium absorbs a portion of the incoming photons and heats up due to the absorbed energy. The heating then results in thermal expansion of the tissue and the expansion generates a pressure field. This process is called the photoacoustic effect. The pressure field generated by the photoacoustic effect can be written as~\cite{BaUh-IP10,FiScSc-PRE07}
\begin{equation}\label{EQ:Data}
	H(\bx) = \Gamma(\bx) \sigma(\bx) u(\bx)
\end{equation}
where $\sigma(\bx) u(\bx)$ is the total energy absorbed locally at $\bx\in\Omega$ and $\Gamma$ is the Gr\"uneisen coefficient that describes the efficiency of the photoacoustic effect. This initial pressure field $H$ then propagates in the form of ultrasound with sound speed $c(\bx)$~\cite{BaUh-IP10,FiScSc-PRE07,StUh-IP09}
\begin{equation}\label{EQ:Acous}
	\begin{array}{rcll}
  	\dfrac{1}{c^2(\bx)} \dfrac{\partial^2 p}{\partial t^2}-\Delta p &=& 0,& \mbox{in}\ \bbR_+\times \Omega\\
	p(0,\bx) &=& H(\bx), &\mbox{in}\ \Omega\\
	\dfrac{\partial p}{\partial t}(0,\bx) &=& 0, &\mbox{in}\ \Omega\\
	p(t,\bx) &=& 0,& \mbox{on}\ \bbR_+\times\partial\Omega
	\end{array}
\end{equation}
assuming that the medium has no acoustic attenuation effect. We refer interested readers to~\cite{BaJoJu-IP10,BaUh-IP10,FiScSc-PRE07} for the derivation and justification of the models~\eqref{EQ:Diff} and ~\eqref{EQ:Acous}.

The time-dependent acoustic signals that arrive on the surface of the medium, $\frac{\partial p}{\partial n}|_{(0,T)\times \partial\Omega}$, are then measured with acoustic devices for a sufficiently long time $T$. From the knowledge of these acoustic measurements, one is interested in reconstructing the diffusion and absorption properties of the medium; see~\cite{Bal-IO12,Beard-IF11,CoLaBe-SPIE09,Kuchment-MLLE12,KuKu-HMMI10,LiWa-PMB09,PaSc-IP07,Scherzer-Book10,Wang-DM04,Wang-IEEE08} for overviews of photoacoustic tomography.

Image reconstructions in PAT are usually done in a two-step process. In the first step, one uses the boundary acoustic signal to reconstruct the initial pressure field $H$ in~\eqref{EQ:Data}, inside the medium. This step has been extensively studied in the past decade under various circumstances; see for instance ~\cite{AgKuKu-PIS09,AgQu-JFA96,AmBrJuWa-LNM12,BuMaHaPa-PRE07,CoArBe-IP07,FiHaRa-SIAM07,Haltmeier-SIAM11B,HaScSc-M2AS05,Hristova-IP09,KiSc-SIAM13,KuKu-EJAM08,Kunyansky-IP08,Nguyen-IPI09,PaSc-IP07,QiStUhZh-SIAM11,StUh-IP09,Tittelfitz-IP12} and references therein. In the second step, usually called the quantitative step, one reconstructs the diffusion coefficient $D(\bx)$, the absorption coefficient $\sigma(\bx)$, and whenever possible, the Gr\"uneisen coefficient $\Gamma(\bx)$ using the internal datum $H(\bx)$. This step has recently received great attention as well; see for instance ~\cite{AmBoJuKa-SIAM10,BaJoJu-IP10,BaUh-IP10,BaUh-CPAM13,CoArKoBe-AO06,CoTaAr-CM11,GaOsZh-LNM12,LaCoZhBe-AO10,MaRe-CMS14,NaSc-SIAM14,PuCoArKaTa-IP14,ReGaZh-SIAM13,SaTaCoAr-IP13,ShCoZe-AO11,Zemp-AO10} and references therein. Uniqueness and stability results for the reconstruction procedures have been established in different settings.

The above two-step process works perfectly fine in the setting where acoustic measurement can be performed everywhere on the boundary $\partial\Omega$ and the ultrasound speed $c(\bx)$ is known. In this case, the initial pressure field $H(\bx)$ can be reconstructed relatively stably in the first step of PAT reconstructions. In the case where acoustic data can only be measured on part of the boundary, i.e. the so called limited-view setting, the reconstruction of $H$ in the first step can be fairly unstable (especially in the part of the domain from where acoustic signals do not travel easily to the measurement locations)~\cite{CoArBe-IP07,HuXiMaWa-JBO13}. In the case where the sound speed $c(\bx)$ is also unknown, which is often the case in real-world applications, the first step reconstruction is also problematic. In this case, one has to reconstruct simultaneously the wave speed $c$ and the initial pressure field $H$. The reconstruction has been shown recently by Stefanov and Uhlmann~\cite{StUh-IPI13}, based on a slightly different formulation of the acoustic problem, to be very \emph{unstable}.

In practical applications, we almost always measure acoustic data generated from multiple optical illuminations. The two-step process, however, does not take advantage of these multiple data sets. The reason is that if we change the illumination source $g$, the initial pressure field $H$ will also be changed. Therefore, every time we add a new measurement, we introduce a new unknown $H$ in the first step of the PAT reconstruction.

The ultimate objectives of the PAT reconstructions are the coefficients $(c, D, \sigma, \Gamma)$. These coefficients do not change when the illumination is changed. Therefore adding more data should not introduce more unknowns in the reconstruction. In other words, instead of reconstructing the unknowns $(c, H)$ and then the coefficients $(D, \sigma, \Gamma)$, we should reconstruct directly from the measured data the coefficients $(c, D, \sigma, \Gamma)$, without the intermediate variable $H$. This way, one can potentially use data sets from multiple illuminations to help stabilize the reconstruction, even in the case of partial measurements (for each illumination).

The aim of the current work is exactly to develop one-step reconstruction strategies that recover the optical coefficients and the wave speed directly from the measured acoustic data. For simplicity of presentation, we focus only on the absorption coefficient $\sigma$ and the wave speed assuming that the diffusion coefficient $D$ and the Gr\"uneisen coefficient $\Gamma$ are known. In this case, we have to invert the nonlinear map $\Lambda(c,\sigma; g)$ defined through the following relation:
\begin{equation}\label{EQ:Data Acous}
	 \frac{\partial p}{\partial n}|_{(0,T)\times \partial\Omega} = \Lambda(c,\sigma;g).
\end{equation}
To the best of our knowledge, there is no uniqueness result on this inverse problem of reconstructing simultaneously the ultrasound speed and the absorption coefficient, besides the special case in~\cite{KiSc-SIAM13}. The result in~\cite{StUh-IPI13} indicates that the reconstruction would be \emph{unstable} when only one data set, i.e. data collected from only one optical illumination, is used. Our main goal here is to show \emph{numerically} that using multiple data sets allows us to obtain fairly stable reconstructions, with appropriate \emph{a priori} bounds on the unknown coefficients. Rigorous mathematical study of the stability of our reconstruction approach will be a future work.

To set up the problem, we assume in the rest of the paper that (A-i) the domain $\Omega$ is bounded with smooth boundary $\partial\Omega$, (A-ii) the boundary condition $g$ is the restriction of a $\cC^4$ function on $\partial\Omega$, (A-iii) the coefficients $c(\bx)\in \cC^4(\bar\Omega)$, $\sigma(\bx)\in C^4(\Omega)$, $D(\bx)\in C^3(\Omega)$, and $\Gamma(\bx)\in C^4(\Omega)$, and (A-iv) $(c, \sigma, D, \Gamma)\in\cA_{\alpha}\times\cA_{\beta}\times\cA_{\zeta}\times\cA_{\xi}$ with
\begin{equation}
\cA_{\alpha}=\{f(\bx): 0< \underline{\alpha} \le f(\bx) \le \overline{\alpha} <\infty,\ \forall\bx\in\Omega\},
\end{equation}
$\underline{\alpha}$ and $\overline{\alpha}$ being two constants. With the assumptions (A-i)-(A-iv), we conclude from standard elliptic theory~\cite{Evans-Book98,GiTr-Book00} that the diffusion equation~\eqref{EQ:Diff} admits a unique solution $u \in \cC^4(\Omega)$. This implies that $\Gamma\sigma u \in\cC^4(\Omega)$. Therefore the wave equation~\eqref{EQ:Acous} admits a unique solution $p\in\cC^4((0,T)\times\Omega)$ following the theory in~\cite{DiDoNaPaSi-SIAM02,Isakov-Book02}. The datum~\eqref{EQ:Data Acous} is then well-defined and can be viewed as a map $\Lambda:\cC^4(\bar\Omega)\times\cC^4(\Omega) \mapsto \cH^{5/2}((0,T)\times\partial\Omega)$.

The rest of the paper is structured as follows. We first present a one-step inversion algorithm in the linearized setting in Section~\ref{SEC:Lin}. We then implement in Section~\ref{SEC:Min} the algorithm for reconstructions in the fully nonlinear setting. We show some numerical reconstructions based on synthetic data in Section~\ref{SEC:Num} to demonstrate the feasibility of the algorithm. Concluding remarks are offered in Section~\ref{SEC:Concl}.

\section{Inversion with Born approximation}
\label{SEC:Lin}

We denote by $J$ the total number of illumination sources available, and $\frac{\partial p_j}{\partial n}|_{(0,T)\times \partial \Omega}$ the measured datum on the boundary for illumination $g_j$ ($1\le j\le J$), in the time interval $(0,T)$. Our objective is thus to \emph{numerically} invert the following nonlinear system to reconstruct $(c,\sigma)$:
\begin{equation}\label{EQ:Nonl Syst}
	\frac{\partial \bp}{\partial \bn}\equiv
	\left(
	\begin{array}{c}
		\frac{\partial p_1}{\partial n}|_{(0,T)\times \partial\Omega}\\
		\vdots\\
		\frac{\partial p_J}{\partial n}|_{(0,T)\times \partial\Omega}
	\end{array}
	\right)
	=\left(
	\begin{array}{c}
		\Lambda(c,\sigma;g_1)\\
		\vdots\\
		\Lambda(c,\sigma;g_J)
	\end{array}
	\right)
	\equiv \bLambda(c,\sigma;\bg) .
\end{equation}
In this section, we develop a one-step algorithm for the inverse problem in the linearized setting where we intend to reconstruct perturbations to the coefficients around known backgrounds. This is useful in some real-world applications where variations of the ultrasound speed and the absorption coefficient are relatively small (for instance, it is well-known that ultrasound speed in biological tissues is very similar to that in water with about $15\%$ variations from tissue to tissue~\cite{YoKaHaYoSoCh-OE12}). 

\subsection{The Born approximation}

We denote by $c_0(\bx)$ and $\sigma_0(\bx)$ respectively the \emph{known} background ultrasound speed and absorption coefficient. We assume that the true coefficients are of the forms
\begin{equation}\label{EQ:Coeff Pert}
	c(\bx) = c_0(\bx) + \tilde c(\bx), \quad\mbox{and}\quad 
\sigma(\bx) = \sigma_0(\bx)+\tilde{\sigma}(\bx),
\end{equation}
where the perturbations $\tilde c$ and $\tilde\sigma$ satisfy $\|\tilde c/c_0\|_{L^\infty(\Omega)} \ll 1$ and $\|\tilde\sigma/\sigma_0\|_{L^\infty(\Omega)}\ll 1$ respectively. The perturbations in the coefficients lead to perturbations in the solutions to the diffusion equation and the wave equation as follows:
\begin{equation}\label{EQ:Sol Pert}
	u_j=u_j^0+\tilde u_j, \quad \mbox{and} \quad 
	p_j(t,\bx) = p_j^0(t,\bx) + \tilde p_j(t,\bx)
\end{equation}
where the background photon densities $u^0_j$ are determined as solutions to
\begin{equation}\label{EQ:Diff Backgr}
\begin{array}{rcll}
-\nabla \cdot D\nabla u^0_j (\bx) + \sigma_0 (\bx) u^0_j (\bx) \, &=& \,  0, &\mbox{in}\ \Omega\\
u_j^0 &=& g_j,  & \mbox{on}\ \partial \Omega
\end{array}
\end{equation}
and the background acoustic pressure fields $p_j^0$ solve the acoustic wave equations:
\begin{equation}\label{EQ:Acous Backgr}
	\begin{array}{rcll}
  	\dfrac{1}{c_0^2(\bx)} \dfrac{\partial^2 p_j^0}{\partial t^2} -\Delta p_j^0 &=& 0, & \mbox{in}\ (0,T)\times \Omega\\
	p_j^0(0,\bx) &=& \Gamma \sigma_0 u_j^0(\bx), &\mbox{in}\ \Omega\\
	\dfrac{\partial p_j^0}{\partial t}(0,\bx) &=& 0, &\mbox{in}\ \Omega\\
	p_j^0(t,\bx) &=& 0, & \mbox{on}\ (0,T)\times\partial\Omega
	\end{array}
\end{equation}

We check that the perturbations $\tilde u_j$ and $\tilde p_j$ solve respectively, after neglecting higher-order terms, the diffusion equation:
\begin{equation}\label{EQ:Diff Sol Pert}
\begin{array}{rcll}
-\nabla \cdot D\nabla \tilde u_j (\bx) + \sigma_0 (\bx) \tilde u_j (\bx) &=& -\tilde \sigma u_j^0, &\mbox{in}\ \Omega\\
\tilde u_j &=& 0,  & \mbox{on}\ \partial \Omega
\end{array}
\end{equation}
and the wave equation:
\begin{equation}\label{EQ:Acous Sol Pert}
	\begin{array}{rcll}
  	\dfrac{1}{c_0^2(\bx)} \dfrac{\partial^2 \tilde p_j}{\partial t^2} -\Delta \tilde p_j &=& \dfrac{2}{c_0^3} \dfrac{\partial^2 p_j^0}{\partial t^2}\tilde{c}(\bx), & \mbox{in}\ (0,T)\times \Omega\\
	\tilde p_j(0,\bx) &=& \Gamma (\tilde{\sigma}u^0_j + \sigma_0 \tilde{u}_j), &\mbox{in}\ \Omega\\
	\dfrac{\partial \tilde p_j}{\partial t}(0,\bx) &=& 0, &\mbox{in}\ \Omega\\
	\tilde p_j(t,\bx) &=& 0, & \mbox{on}\ (0,T)\times\partial\Omega
	\end{array}
\end{equation}

Let us now introduce the linear operators $\Lambda_c^j(c_0,\sigma_0)$ and $\Lambda_\sigma^j(c_0,\sigma_0)$ through:
\begin{equation}\label{EQ:Frechet D}
	\frac{\partial \tilde p_j^c}{\partial n}|_{(0,T)\times\partial\Omega}=\Lambda_c^j(c_0,\sigma_0)\tilde c, \quad 
	\mbox{and}\quad 
	\frac{\partial \tilde p_j^\sigma}{\partial n}|_{(0,T)\times\partial\Omega}=\Lambda_\sigma^j(c_0,\sigma_0)\tilde \sigma,
\end{equation}
with $\tilde p_j^c$ and $\tilde p_j^\sigma$ respectively the solutions to
\begin{equation}\label{EQ:Acous Sol Pert c}
	\begin{array}{rcll}
  	\dfrac{1}{c_0^2(\bx)} \dfrac{\partial^2 \tilde p_j^c}{\partial t^2} -\Delta \tilde p_j^c &=& \dfrac{2}{c_0^3} \dfrac{\partial^2 p_j^0}{\partial t^2}\tilde{c}(\bx), & \mbox{in}\ (0,T)\times \Omega\\
	\tilde p_j^c(0,\bx) &=& 0, &\mbox{in}\ \Omega\\
	\dfrac{\partial \tilde p_j^c}{\partial t}(0,\bx) &=& 0, &\mbox{in}\ \Omega\\
	\tilde p_j^c(t,\bx) &=& 0, & \mbox{on}\ (0,T)\times\partial\Omega
	\end{array}
\end{equation}
and
\begin{equation}\label{EQ:Acous Sol Pert sigma}
	\begin{array}{rcll}
  	\dfrac{1}{c_0^2(\bx)} \dfrac{\partial^2 \tilde p_j^\sigma}{\partial t^2} -\Delta \tilde p_j^\sigma &=& 0, & \mbox{in}\ (0,T)\times \Omega\\
	\tilde p_j^\sigma(0,\bx) &=& \Gamma (\tilde{\sigma}u^0_j + \sigma_0 \tilde{u}_j), &\mbox{in}\ \Omega\\
	\dfrac{\partial \tilde p_j^\sigma}{\partial t}(0,\bx) &=& 0, &\mbox{in}\ \Omega\\
	\tilde p_j^\sigma(t,\bx) &=& 0, & \mbox{on}\ (0,T)\times\partial\Omega
	\end{array}
\end{equation}
We can then write the perturbation of the datum as
\begin{equation}\label{EQ:Data Split}
	\frac{\partial \tilde p_j}{\partial n}(t,\bx)|_{(0,T)\times\partial\Omega}=\Lambda_c^j(c_0,\sigma_0)\tilde c + \Lambda_\sigma^j(c_0,\sigma_0) \tilde \sigma.
\end{equation}
We then collect data for all $J$ sources to get a system of equations for the unknowns:
\begin{equation}\label{EQ:Lin Sys}
	\bLambda_{c,\sigma}(c_0,\sigma_0)
	\left(
	\begin{array}{c}
	\tilde c\\
	\tilde \sigma
	\end{array}
	\right)
	=\frac{\partial\tilde\bp}{\partial\bn}
\end{equation}
with
\begin{equation}
	\bLambda_{c,\sigma}(c_0,\sigma_0)=\left(
	\begin{array}{cc}
	\Lambda_c^1(c_0,\sigma_0) & \Lambda_\sigma^1(c_0,\sigma_0)\\
	\vdots & \vdots\\
	\Lambda_c^J(c_0,\sigma_0) & \Lambda_\sigma^J(c_0,\sigma_0)
	\end{array}
	\right),
	\qquad
	\frac{\partial\tilde\bp}{\partial\bn}=
	\left(
	\begin{array}{c}
	\frac{\partial \tilde p_1}{\partial n}\\
	\vdots\\
	\frac{\partial \tilde p_J}{\partial n}
	\end{array}
	\right) .
\end{equation}

This is the Born approximation of the original nonlinear problem~\eqref{EQ:Nonl Syst}. The approximation can be justified as stated in the following lemma.
\begin{lemma}\label{LMMA:Derivative}
Let $\Omega$, $D$, $\Gamma$ and $g_j$ satisfy the assumptions in (A-i)-(A-iv). Then the datum generated from $g_j$, viewed as the map:
\begin{equation}
	\Lambda(c,\sigma;g_j):
	\begin{array}{ccl}
		(c,\sigma) &\mapsto& \dfrac{\partial p_j}{\partial n}|_{(0,T)\times\partial\Omega}\\
		\cC^4(\bar\Omega)\times \cC^4(\Omega) &\mapsto & \cH^{1/2}((0,T)\times\partial\Omega)
	\end{array}
	\end{equation}
	is Fr\'echet differentiable at any $(c_0,\sigma_0)\in \cC^4(\bar\Omega)\times \cC^4(\Omega)$ that satisfies the assumption in (A-iv). The derivative at $(c_0,\sigma_0)$ in the direction $(\tilde c, \tilde \sigma)\in \cC^4(\bar\Omega)\times \cC^4(\Omega)$ (such that $c_0+\tilde c$ and $\sigma_0+\tilde\sigma$ satisfy (A-iv)) is $(\Lambda_c^j(c_0,\sigma_0)\tilde c, \Lambda_\sigma^j(c_0,\sigma_0)\tilde\sigma)$ as defined in~\eqref{EQ:Frechet D}.
\end{lemma}
The lemma can be proven using standard arguments such as those in~\cite{BaSy-CPDE96,DiDoNaPaSi-SIAM02,Isakov-Book02,Rakesh-CPDE88}. We provide a sketch of the proof in the Appendix. Note that even though $\Lambda(c,\sigma;g_j)$ is well-defined as a map $\Lambda(c,\sigma;g_j): \cC^4(\bar\Omega)\times \cC^4(\Omega)  \mapsto  \cH^{5/2}((0,T)\times\partial\Omega)$, we can only prove its differentiability as a map $\Lambda(c,\sigma;g_j): \cC^4(\bar\Omega)\times \cC^4(\Omega)  \mapsto  \cH^{1/2}((0,T)\times\partial\Omega)$.

We now need to invert~\eqref{EQ:Data Split} (or ~\eqref{EQ:Lin Sys} when multiple data sets are available) to reconstruct the perturbation $(\tilde c, \tilde \sigma)$. It is well-known that with a single measurement $\frac{\partial \tilde p_j}{\partial n}(t,\bx)|_{(0,T)\times\partial\Omega}$, one could reconstruct one of $\tilde c$ and $\tilde H_j=\Gamma (\tilde{\sigma}u_j^0 + \sigma_0 \tilde{u}_j)$ (thus $\tilde \sigma$ since $\tilde H_j$ uniquely determines $\tilde\sigma$~\cite{BaRe-IP11,BaUh-IP10}) assuming that the other is known~\cite{Isakov-Book02,StUh-TAMS13,Yamamoto-JIIP94,Yamamoto-IP95}. In a slightly different  setting, Stefanov and Uhlmann~\cite{StUh-IPI13} showed that even if a single measurement $\frac{\partial \tilde p_j}{\partial n}(t,\bx)|_{(0,T)\times\partial\Omega}$ is enough to reconstruct the pair $(\tilde c, \tilde H_j)$ (and thus $(\tilde c, \tilde\sigma)$) uniquely, the reconstruction would be extremely \emph{unstable}~\cite[Theorem 1]{StUh-IPI13}. Our hope here is that, without introducing new unknowns (since we reconstruct $\tilde \sigma$ directly, not $\tilde H_j$), by using multiple data sets we can improve the stability of the reconstruction, assuming that uniqueness can be achieved.

\subsection{Reconstruction based on Born approximation}

To invert the linear system~\eqref{EQ:Lin Sys}, we use the technique of Landweber iteration~\cite{Kirsch-Book11}. The iteration takes the form:
\begin{equation}
	\left(
	\begin{array}{c}
	\tilde c_{k+1}\\
	\tilde \sigma_{k+1}
	\end{array}
	\right)
	=(\bI-\tau\bLambda_{c,\sigma}^*(c_0,\sigma_0)\bLambda_{c,\sigma}(c_0,\sigma_0)) 	
	\left(
	\begin{array}{c}
	\tilde c_{k}\\
	\tilde \sigma_{k}
	\end{array}
	\right)
	+\tau 	\bLambda_{c,\sigma}^*(c_0,\sigma_0)\frac{\partial \tilde\bp}{\partial \bn},\quad k\ge 0,
\end{equation}
with a reasonable given initial guess. The parameter $\tau$, $0<\tau<2/\Sigma^2$ with $\Sigma$ being the largest singular value of $\bLambda_{c,\sigma}(c_0,\sigma_0)$, is a positive algorithmic parameter that we select by trial and error since we do not have good estimates on the singular values of $\bLambda_{c,\sigma}(c_0,\sigma_0)$. The components of the adjoint operator $\bLambda_{c,\sigma}^*(c_0,\sigma_0)$
\begin{equation}
	\bLambda_{c,\sigma}^*(c_0,\sigma_0) = 
	\left(
	\begin{array}{ccc}
	\Lambda_c^{1*}(c_0,\sigma_0) & \cdots & \Lambda_c^{J*}(c_0,\sigma_0)\\
	\Lambda_\sigma^{1*}(c_0,\sigma_0) & \cdots & \Lambda_\sigma^{J*}(c_0,\sigma_0)
	\end{array}
	\right),
\end{equation}
are given as follows. The adjoint operator $\Lambda_c^{j*}(c_0,\sigma_0)$, in the sense of 
\[
\aver{\Lambda_c^j(c_0,\sigma_0) \tilde c, y_j}_{L^2((0,T)\times\partial\Omega)}=\aver{\tilde c,\Lambda_c^{j*}(c_0,\sigma_0)y_j}_{L^2(\Omega)},\ \ \forall y_j\in L^2((0,T)\times\partial\Omega),
\]
is given as
\begin{equation}
	\Lambda_c^{j*}(c_0,\sigma_0) y_j = \int_0^T \dfrac{2}{c_0^3} \dfrac{\partial^2 p_j^0}{\partial t^2}(t,\bx) \tilde q_j(t,\bx) dt  
\end{equation}
with $\tilde q_j$ the solution to the following adjoint wave equation:
\begin{equation}\label{EQ:Acous Sol Pert Adj}
	\begin{array}{rcll}
  	\dfrac{1}{c_0^2(\bx)} \dfrac{\partial^2 \tilde q_j}{\partial t^2} -\Delta \tilde q_j &=& 0, & \mbox{in}\ (0,T)\times \Omega\\
	\tilde q_j(T,\bx) &=& 0, &\mbox{in}\ \Omega\\
	\dfrac{\partial \tilde q_j}{\partial t}(T,\bx) &=& 0, &\mbox{in}\ \Omega\\
	\tilde q_j(t,\bx) &=& -y_j, & \mbox{on}\ (0,T)\times\partial\Omega
	\end{array}
\end{equation}
The adjoint operator $\Lambda_\sigma^{j*}(c_0,\sigma_0)$, in the sense of 
\[
\aver{\Lambda_\sigma^j(c_0,\sigma_0) \tilde \sigma, z_j}_{L^2((0,T)\times\partial\Omega)}=\aver{\tilde \sigma,\Lambda_\sigma^{j*}(c_0,\sigma_0)z_j}_{L^2(\Omega)},\ \ \forall z_j\in L^2((0,T)\times\partial\Omega),
\]
is given as
\begin{equation}
	\Lambda_\sigma^{j*} z_j = -(\dfrac{1}{c_0^2}\Gamma \dfrac{\partial\tilde q_j}{\partial t}(0,\bx)+\tilde v_j)u_j^0,
\end{equation}
where $\tilde q_j$ is now the solution to~\eqref{EQ:Acous Sol Pert Adj} with $y_j$ replaced by $z_j$ and $\tilde v_j$ is the solution to the adjoint diffusion equation:
\begin{equation}\label{EQ:Diff Sol Pert Adj}
\begin{array}{rcll}
-\nabla \cdot D\nabla \tilde v_j (\bx) + \sigma_0 (\bx) \tilde v_j (\bx) &=& -\dfrac{1}{c_0^2}\Gamma\sigma_0 \dfrac{\partial\tilde q_j}{\partial t}(0,\bx), &\mbox{in}\ \Omega\\
\tilde v_j &=& 0,  & \mbox{on}\ \partial \Omega
\end{array}
\end{equation}
Note that even though we split the components for $c$ and $\sigma$ in the datum in the form of ~\eqref{EQ:Data Split} for the convenience of presentation, we do not need to solve two wave equations, ~\eqref{EQ:Acous Sol Pert c} and ~\eqref{EQ:Acous Sol Pert sigma}, to compute the datum. Instead we need only to solve one wave equation, i.e~\eqref{EQ:Acous Sol Pert}. Therefore, in each iteration of the Landweber algorithm, we need to solve $J$ forward wave equations and $J$ forward diffusion equations to evaluate $\bff_k\equiv\bLambda_{c,\sigma}(c_0,\sigma_0)\left(\begin{array}{c}\tilde c_k\\ \tilde \sigma_k \end{array}\right)$ and then $J$ adjoint wave equations and $J$ adjoint diffusion equations to evaluate $\bLambda_{c,\sigma}^*(c_0,\sigma_0)\bff_k$. The last term in the iteration does not change during the iteration, so it needs only to be computed once before the iteration starts.

\section{One-step nonlinear reconstruction}
\label{SEC:Min}

To solve the full nonlinear inverse problem, we take an optimal control approach. We look for the solution to the inverse problem as 
\begin{equation}
	\left(
	\begin{array}{c}
	c_{\rm min}\\
	\sigma_{\rm min}
	\end{array}
	\right)=\argmin_{(c,\sigma)\in\cA_{\alpha}\times\cA_{\beta}} F(c,\sigma),
\end{equation}
with the objective functional $F$ given as:
\begin{equation}\label{EQ:Obj}
F(c, \sigma) = \frac{1}{2} \|\bLambda(c,\sigma;\bg)-\frac{\partial \bp}{\partial \bn}^\#\|^2_{[L^2((0,T)\times\partial\Omega)]^J},
\end{equation}
where $\frac{\partial \bp}{\partial \bn}^\#$ is the collection of measured acoustic data.

We implemented the Levenberg-Marquardt method~\cite{Kelley-Book99,NoWr-Book99} to solve the minimization problem. The method is characterized with the following iteration:
\begin{equation}\label{EQ:LM Iter}
	\left(
	\begin{array}{c}
	c_{k+1}\\
	\sigma_{k+1}
	\end{array}
	\right)
	=	
	\left(
	\begin{array}{c}
	c_{k}\\
	\sigma_{k}
	\end{array}
	\right)
	-(\bLambda_{c,\sigma}^*(c_k,\sigma_k)\bLambda_{c,\sigma}(c_k,\sigma_k)+\mu_k\bI)^{-1} \bLambda_{c,\sigma}^*(c_k,\sigma_k)\bz_k,\quad k\ge 0,
\end{equation}
where $\mu_k$ is an algorithm parameter, $\bz_k$ is the residual at step $k$:
\[
\bz_k=\bLambda(c_k,\sigma_k;\bg)-\frac{\partial \bp}{\partial \bn}^\#,
\]
and $\bLambda_{c,\sigma}(c_k,\sigma_k)$ is the Fr\'echet derivative of $\bLambda$ at $(c_k,\sigma_k)$. In our implementation, we take the Levenberg-Marquardt parameter $\mu_k$ as a small constant, although we are aware that there are principles in the literature to guide the selection of this parameter in a more ``optimal'' way; see for instance~\cite{Kelley-Book99,NoWr-Book99}.

The Levenberg-Marquardt algorithm~\eqref{EQ:LM Iter} requires the inverse of the operator $(\bLambda_{c,\sigma}^*\bLambda_{c,\sigma}+\mu_k\bI)$ at each iteration. In our implementation, we do not form the operator (which in discrete case is a matrix) explicitly and then invert it, since this would require large computer memory to store the matrix. Instead, we use a matrix-free approach to save memory in the following way. For any function (which in discrete case is a vector) $\by$, to compute $\bz=(\bLambda_{c,\sigma}^*\bLambda_{c,\sigma}+\mu_k\bI)^{-1}\by$, we solve the following equation:
\begin{equation}
(\bLambda_{c,\sigma}^*\bLambda_{c,\sigma}+\mu_k\bI)\bz = \by .
\end{equation}
This is a symmetric positive definite problem which we solve with a standard conjugate gradient method~\cite{Saad-Book03}. The conjugate gradient method does not require the explicit form of the operator $(\bLambda_{c,\sigma}^*\bLambda_{c,\sigma}+\mu_k\bI)$, but only its action on given functions (vectors in discrete case). For any given $\bz$, we solve $J$ forward wave equations and $J$ forward diffusion equations to evaluate $\bff\equiv\bLambda_{c,\sigma}(c_k,\sigma_k)\bz$ and then $J$ adjoint wave equations and $J$ adjoint diffusion equations to evaluate $\bLambda_{c,\sigma}^*(c_k,\sigma_k)\bff$.

To impose the bound constraints on the coefficients, that is $c\in \cA_\alpha$ (i.e. $\underline{\alpha}\le c(\bx) \le \overline{\alpha}$) and $\sigma\in\cA_\beta$ (i.e. $\underline{\beta}\le \sigma(\bx) \le \overline{\beta}$), we implement the algorithm for the new variables $\upsilon(\bx)$ and $\eta(\bx)$ that are related respectively to $c(\bx)$ and $\sigma(\bx)$ through the relations:
\begin{equation}
	c(\bx)= \frac{\underline{\alpha}+\overline{\alpha}}{2}+\frac{\overline{\alpha}-\underline{\alpha}}{2}\tanh \upsilon(\bx), \qquad \sigma(\bx)= \frac{\underline{\beta}+\overline{\beta}}{2}+\frac{\overline{\beta}-\underline{\beta}}{2}\tanh\eta(\bx), \qquad 
\end{equation}
The Fr\'echet derivatives of $\bLambda$ with respect to the new variables at $(\upsilon_0, \eta_0)$ in the direction $(\tilde \upsilon, \tilde \eta)$ can be computed using the chain rule straightforwardly as
\begin{equation}
\bLambda_{\upsilon}(\upsilon_0)\tilde \upsilon =\bLambda_{c}(c(\upsilon_0)) [\frac{\overline{\alpha}-\underline{\alpha}}{2} \sech^2\upsilon_0(\bx) \tilde \upsilon],\quad
\bLambda_{\eta}(\eta_0)\tilde \eta =\bLambda_{\sigma}(\sigma(\eta_0)) [\frac{\overline{\beta}-\underline{\beta}}{2} \sech^2\eta_0(\bx) \tilde \eta].
\end{equation}
Note that we have chosen different bounds for $c$ (i.e $\overline{\alpha}$ and $\underline{\alpha}$) and $\sigma$ (i.e. $\overline{\beta}$ and $\underline{\beta}$) since in practice the two functions have different ranges of values.

\section{Numerical experiments}
\label{SEC:Num}

We now present some numerical experiments to check the performance of the reconstruction algorithm. To simplify the presentation, we non-dimensionalize the problem so that all the numbers presented below have no dimensions. We consider the problem in the square domain $\Omega=[0\ 2]\times[0\ 2]$ with constant diffusion coefficient $D=0.02$ and Gr\"uneisen coefficient $\Gamma=1$.

In our implementation, we discretize the forward and adjoint diffusion equations, for instance ~\eqref{EQ:Diff} and~\eqref{EQ:Diff Sol Pert Adj},  with a first-order finite element method. We discretize the forward and adjoint wave equations, for instance~\eqref{EQ:Acous} and~\eqref{EQ:Acous Sol Pert Adj}, with a standard second-order finite difference scheme on a uniform grid. The finite element discretization of the diffusion equations is performed on a triangle mesh that shares the same nodes as the uniform grid for the wave equation. This way we do not need to interpolate between the two types of grids when using quantities from the diffusion solution in the wave equations or vice versa.

We will perform numerical reconstructions in both the nonlinear and the linearized settings. To generate synthetic data for the nonlinear inversions, we solve the diffusion equation~\eqref{EQ:Diff} and the wave equation~\eqref{EQ:Acous} with the true absorption coefficient and ultrasound speed, and then compute the data using~\eqref{EQ:Data Acous}. When adding multiplicative random noise to the datum $\frac{\partial p}{\partial n}$ we perform the transformation $\frac{\partial p}{\partial n}\to \frac{\partial p}{\partial n}\times(1+\frac{\kappa}{100}\cdot {\rm rand})$ where {\rm rand} is a uniform random variable with mean $0$ and variance $1$ (and thus range $[-\sqrt{3}, \sqrt{3}]$). We use $\kappa$ to measure the level of noise in the data. The values of $\kappa$ will be given in the simulations we present below. For the linearized inversions, we construct synthetic data directly using the linearized model~\eqref{EQ:Data Split} for a given perturbation of the coefficients, $(\tilde c,\tilde \sigma)$. \emph{This means that the data for the linearized inversions are exact. These data do not contain information on the accuracy of the linearized problem as an approximation to the true nonlinear problem}. The linearized inversion simulations we show below will only provide information about the invertability and stability of the linearized inverse problem. 

We measure the quality of the reconstruction with the maximal relative error. For parameter $\frak p$, the error is defined as $\|(\frak p^r-\frak p^t)/\frak p^t\|_{L^\infty}$ where $\frak p^t$ is the true coefficient while $\frak p^r$ is the reconstructed coefficient.

\paragraph{Numerical experiment 1.} The first numerical experiment is devoted to the reconstruction of the absorption coefficient assuming that the ultrasound speed is known. The true absorption coefficient is taken as
\begin{equation}\label{EQ:Abso Coeff}
	\sigma(\bx)=\left\{
	\begin{array}{rl}
		0.15, & \bx\in[0.5\ 1.5]\times[0.5\ 1.5]\\
		0.10, & \bx\in\Omega\backslash[0.5\ 1.5]\times[0.5\ 1.5] .
	\end{array}\right.
\end{equation}
We performed nonlinear reconstructions with the Levenberg-Marquardt algorithm using three types of data: (i) noise-free data ($\kappa=0.0$), (ii) noisy data with $\kappa=0.5$, and (iii) noisy data with $\kappa=1.0$. The data are collected from eight different optical illuminations that are line sources supported on each side of the domain with spatially varying strengths such as those used in~\cite{BaRe-IP11}. The results of the reconstructions are shown in Fig.~\ref{FIG:Mua}. The maximal relative errors in the reconstructions are $0.15$, $0.28$ and $0.64$ respectively. The quality of the reconstructions is very similar to those published in the literature~\cite{AmBoJuKa-SIAM10,CoArKoBe-AO06,CoTaAr-CM11,GaOsZh-LNM12,LaCoZhBe-AO10,MaRe-CMS14,NaSc-SIAM14,PuCoArKaTa-IP14,ReGaZh-SIAM13,SaTaCoAr-IP13,ShCoZe-AO11,Zemp-AO10}. We observe in our numerical experiments that the reconstructions in this case are very robust to changes in initial guesses. Moreover, we impose very loose bounds on the absorption coefficient in the reconstructions: $\underline{\beta}=0 \le \sigma \le 1.0=\overline{\beta}$. Our experience is that this bound is not necessary at all when $\sigma$ is the only unknown to be reconstructed.
\begin{figure}[!ht]
\centering
\includegraphics[angle=0,width=0.90\textwidth]{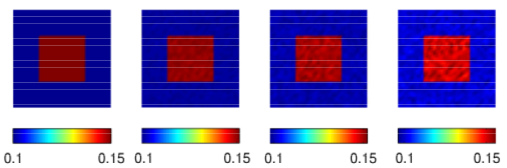} 
\caption{Reconstructions of the absorption coefficient with \emph{known} ultrasound speed. Left to right: true $\sigma(\bx)$ given in~\eqref{EQ:Abso Coeff}, $\sigma$ reconstructed from noise-free data, $\sigma$ reconstructed from noisy data with $\kappa=0.5$, and $\sigma$ reconstructed from noisy data with $\kappa=1.0$.}
\label{FIG:Mua}
\end{figure}

\paragraph{Numerical experiment 2.} In the second numerical experiment, we reconstruct the ultrasound speed $c(\bx)$ assuming that the absorption coefficient is known. We consider the problem with the true ultrasound speed
\begin{equation}\label{EQ:Speed}
	c(\bx)=1.0+0.2 \times \exp(-\dfrac{|\bx-(1,1)|^2}{2\times 0.5^2}),\quad \bx\in\Omega.
\end{equation}
We again performed reconstructions with the three types of data (i)-(iii). The results are shown in Fig.~\ref{FIG:Speed}. We imposed the bound $\underline{\alpha}=0.8\le c \le 1.3=\overline{\alpha}$ on the unknown. The maximal relative errors in the reconstructions are respectively  $0.16$, $0.30$ and $0.57$. The initial guess for all the reconstructions shown is $c=0.9$, although we observe in our numerical experiments that most constant initial guesses within the bounds lead to almost identical final reconstructions.
\begin{figure}[!ht]
\centering
\includegraphics[angle=0,width=0.90\textwidth]{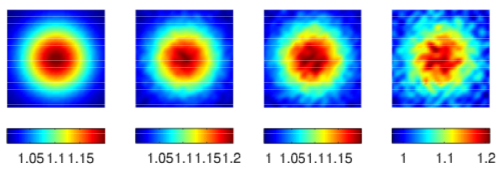} 
\caption{Reconstructions of the ultrasound speed with \emph{known} absorption coefficient. Left to right: true $c(\bx)$ given in~\eqref{EQ:Speed}, $c$ reconstructed from noise-free data ($\kappa=0.0$), $c$ reconstructed from noisy data with $\kappa=0.5$, and $c$ reconstructed from noisy data with $\kappa=1.0$.}
\label{FIG:Speed}
\end{figure}
The data used in the reconstructions are again collected from 8 different optical illuminations. Even though we have more blurring in the reconstructed images, the overall quality of the reconstructions is reasonable, considering that we do not need very strict bounds on the unknown to get these reconstructions.

\paragraph{Numerical experiment 3.} In this experiment, we reconstruct, under the Born approximation, the sound speed $c$ and the absorption coefficient $\sigma$ as shown in the first column of Fig.~\ref{FIG:Simul1-A}. The background of the linearization is $(c_0,\sigma_0)=(1.0,0.1)$. We again use data collected from eight different illuminations. Shown are the reconstructions from noise-free data ($\kappa=0.0$, column II), the reconstructions from noisy data with $\kappa=0.5$ (column III) and the reconstructions from noisy data with $\kappa=1.0$ (column IV). The maximal relative errors in the reconstructions are $(0.40, 0.90)$, $(0.80, 1.20)$, and $(1.05, 2.02)$ respectively. We started all the Landweber iterations with initial condition $(\tilde c,\tilde \sigma)=(0,0)$.
\begin{figure}[!ht]
\centering
\includegraphics[angle=0,width=0.90\textwidth]{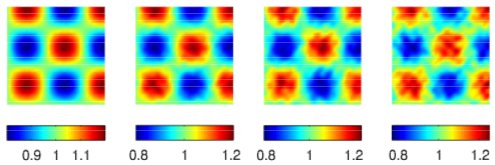}\\
\includegraphics[angle=0,width=0.90\textwidth]{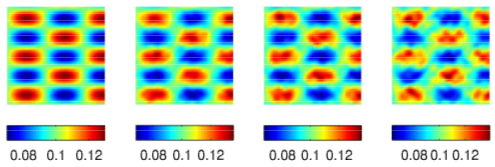}
\caption{Linearized reconstructions of the coefficients $c$ (top) and $\sigma$ (bottom). Shown from left to right are: the true $(c,\sigma)$, the reconstructions from noise-free data, the reconstructions from noisy data with $\kappa=0.5$, the reconstructions from noisy data with $\kappa=1.0$.}
\label{FIG:Simul1-A}
\end{figure}
Let us emphasize again that the synthetic data used in this experiment are constructed directly from the linearized model~\eqref{EQ:Data Split}. Therefore, the data are exact (besides the artificial noise added to them) in the sense that the error in the approximation to the true nonlinear problem is completely neglected. What we are interested in studying is the stability of the linearized inverse problem, not the accuracy of the Born approximation. This is why we can consider fairly large perturbations to the background here.

\paragraph{Numerical experiment 4.} We repeat the reconstructions in the previous experiment (i.e. Numerical experiment 3) in the nonlinear setting with the Levenberg-Marquardt algorithm. The results are shown in Fig.~\ref{FIG:Simul1-B}. The maximal relative errors in the reconstructions are $(0.42, 0.69)$, $(0.61, 1.02)$, and $(0.89, 1.87)$ respectively for data with noise level $\kappa=0.0$, $\kappa=0.5$ and $\kappa=1.0$. In all the reconstructions, we use the initial guess of $(c,\sigma)=(0.9,0.09)$ and the Levenberg-Marquardt algorithm is stopped after $80$ iterations. We impose the bounds $\underline{\alpha}=0.7\le c \le 1.3=\overline{\alpha}$, $\underline{\beta}=0.07\le \sigma \le 0.15=\overline{\beta}$.
\begin{figure}[!ht]
\centering
\includegraphics[angle=0,width=0.90\textwidth]{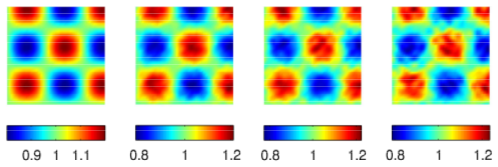}\\
\includegraphics[angle=0,width=0.90\textwidth]{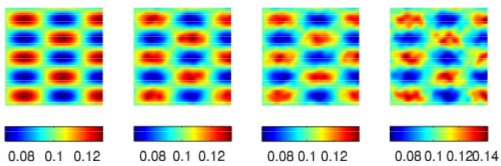}
\caption{Same as Fig.~\ref{FIG:Simul1-A} except that the reconstructions are performed in the nonlinear setting using the Levenberg-Marquardt algorithm.}
\label{FIG:Simul1-B}
\end{figure}

\paragraph{Numerical experiment 5.} In the last numerical experiment, we reconstruct the sound speed $c$ and the absorption coefficient $\sigma$ as shown in the first column of Fig.~\ref{FIG:Simul2}. We again use data collected from eight different illuminations. Shown are the reconstructions from noise-free data ($\kappa=0.0$, column II), the reconstructions from noisy data with $\kappa=0.5$ (column III) and the reconstructions from noisy data with $\kappa=1.0$ (column IV). The maximal relative errors in the reconstructions are $(0.97, 0.56)$, $(1.17, 0.88)$, and $(1.41, 0.99)$ respectively. The Levenberg-Marquardt algorithm is stopped at iteration $100$ in all the reconstructions in this case and the initial guess for all the reconstructions is $(c,\sigma)=(1.0, 0.11)$. We impose stricter bounds on $c$ in this case to have the algorithm converge to reasonable solutions: $\underline{\alpha}=0.85\le c \le 1.2=\overline{\alpha}$, $\underline{\beta}=0.1\le \sigma \le 0.2=\overline{\beta}$. Note that the bounds on the absorption coefficient coincide with the bounds on the true $\sigma$ value. If the constant initial guess is beyond the bounds, the algorithm usually does not converge.
\begin{figure}[!ht]
\centering
\includegraphics[angle=0,width=0.90\textwidth]{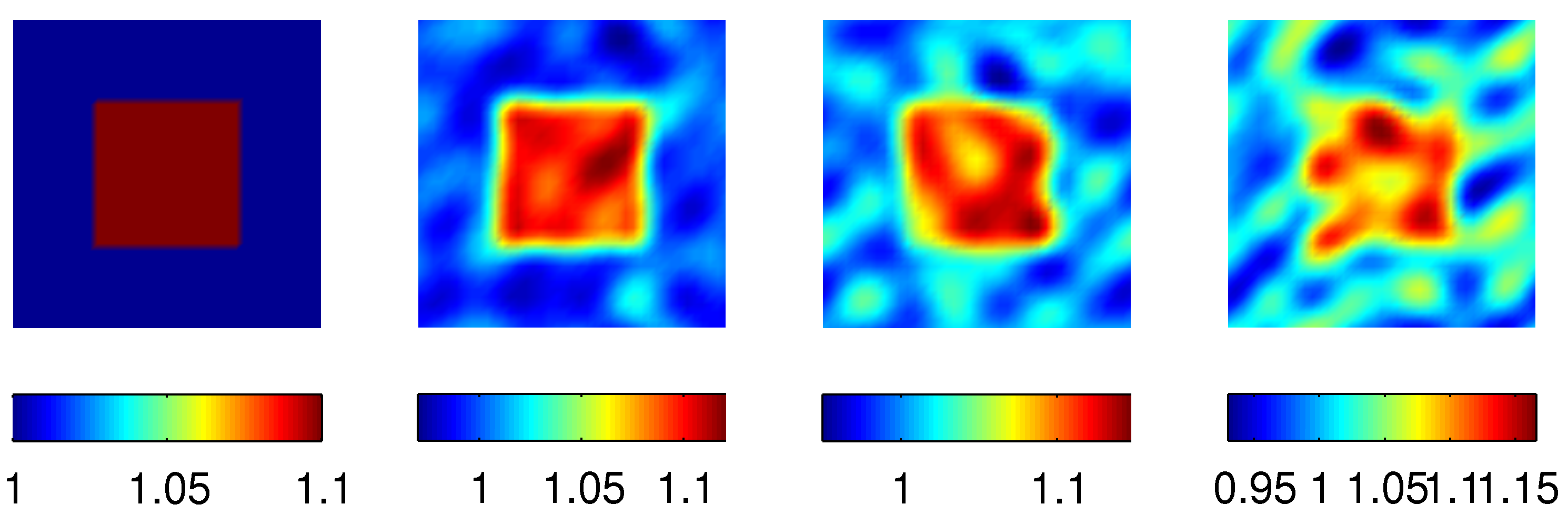}\\
\includegraphics[angle=0,width=0.90\textwidth]{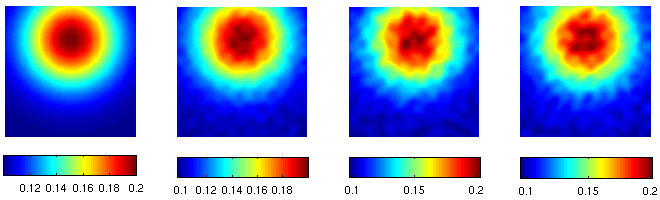}
\caption{Nonlinear reconstructions of the coefficients $c$ (top) and $\sigma$ (bottom). Shown from left to right are: the true $(c,\sigma)$, the reconstructions from noise-free data ($\kappa=0.0$), the reconstructions from noisy data with $\kappa=0.5$, the reconstructions from noisy data with $\kappa=1.0$.}
\label{FIG:Simul2}
\end{figure}

We observe from all the simulations presented in this section that the one-step reconstruction strategy performs reasonably well in either the linearized or the nonlinear case, with \emph{a priori} bounds imposed on the unknowns. Tuning algorithmic parameters, such as the $\tau$ in the Landweber iteration and the $\mu_k$ in the Levenberg-Marquardt algorithm, the maximal number of iterations allowed for the iterative algorithms, and the initial guesses etc, can certainly improve the quality of the reconstructions slightly as we observed in some of the simulations we performed. We did not pursue in that direction. The use of data from even more illuminations would certainly help at least to reduce the average noise level in the data. We did not pursue in that direction either. The simulations we present here provide us a rough idea about the quality of the reconstructions that we could get.

\section{Concluding remarks}
\label{SEC:Concl}

There have been several works in recent years on the simultaneous reconstruction of the ultrasound speed and the optical properties~\cite{JiWa-PMB06,KiSc-SIAM13,LiUh-arXiv15,MaWaWaAn-SPIE15,ScAn-JOSA11,StUh-IPI13,Treeby-JBO13,TrZhThCo-UMB11,ZhAn-SPIE06} in QPAT. Every method proposed attempts to follow the two-step philosophy, that is to first reconstruct ultrasound speed and the initial pressure field and then reconstruct the optical properties. In some cases, additional ultrasound measurements are taken to supplement the photoacoustic measurements~\cite{MaWaWaAn-SPIE15}. We proposed here a reconstruction strategy that combines the two-step reconstruction process into a one-step process to reconstruct directly the sound speed and the optical properties without reconstructing the initial pressure field, which is an intermediate variable. 

The main advantage of our method is that it allows the use of data sets from multiple illuminations which can stabilize the reconstruction when the ultrasound speed is treated as a unknown. When the intermediate variable, the initial pressure field, is to be reconstructed as in~\cite{JiWa-PMB06,KiSc-SIAM13,LiUh-arXiv15,ScAn-JOSA11,StUh-IPI13,Treeby-JBO13,TrZhThCo-UMB11,ZhAn-SPIE06}, it changes with illuminations. Therefore adding data from more illuminations simply adds more unknowns in the reconstruction process. While in our case, adding more data sets does not add more unknowns since $(c,\sigma)$ do not change with illuminations.

The results in~\cite{BaRe-IP11,BaUh-IP10,Isakov-Book02,StUh-TAMS13,Yamamoto-JIIP94,Yamamoto-IP95} show that one can reconstruct uniquely and stably either the ultrasound speed or the optical absorption coefficient if the other one is known. We are not aware of any uniqueness result on the simultaneous reconstruction of both coefficients besides in special cases such as these in~\cite{KiSc-SIAM13,LiUh-arXiv15}. Numerical simulations with synthetic data in Section~\ref{SEC:Num} show that the simultaneous reconstruction is in fact relatively stable when multiple data sets are used, with appropriate \emph{a priori} bounds imposed on the unknowns. This does not contradict the theoretical results in~\cite{StUh-IPI13} on the instability of the problem with a single measurement. We are currently conducting theoretical studies on the stability property of our method. Results will be reported elsewhere.

Let us finish this paper with two additional remarks. First, there have been similar one-step algorithms in the literature~\cite{HaShZe-BOE13,ShHaZe-BOE12} that reconstruct directly optical properties from photoacoustic data. These algorithms are different from the one we proposed here since they all treat the ultrasound speed as \emph{known}. Second, in our formulation, we have assumed that both the diffusion coefficient $D$ and the Gr\"uneisen coefficient $\Gamma$ are known. We can easily modify the algorithm to include $D$ or $\Gamma$ as a unknown in the reconstructions as well. Note that it has been proved~\cite{BaRe-IP11} that one can not reconstruct simultaneously $D$, $\sigma$ and $\Gamma$ with mono-wavelength optical illuminations. Therefore, it is not desired to reconstruct $(c, D, \sigma, \Gamma)$ in the one-step algorithm. However, one can indeed attempt to reconstruct simultaneously $c$, $\sigma$ and $\Gamma$, for instance.

\section*{Acknowledgments}

We would like to thank the anonymous referees for their useful comments, especially for pointing out the references~\cite{HaShZe-BOE13,ShHaZe-BOE12}, that help us improve the quality of this paper. We would also like to thank Professor Hongyu Liu for pointing out the reference~\cite{LiUh-arXiv15}. This work is partially supported by the National Science Foundation through grant DMS-1321018, and the University of Texas through a Moncrief Grand Challenge Faculty Award.

\section*{Appendix: Fr\'echet differentiability of $\Lambda(c,\sigma;g_j)$}

We provide a brief proof of the differentiability of $\Lambda(c,\sigma;g_j)$ as stated in Lemma~\ref{LMMA:Derivative}. To simplify the presentation, we will use the short notation $\omega=\frac{1}{c^2}$, and therefore $\omega_0=\frac{1}{c_0^2}$, $\tilde\omega=-\frac{2\tilde c}{c_0^3}$. We will show differentiability of $\Lambda(\omega,\sigma;g_j)$ with respect to $(\omega,\sigma)$.
\begin{proof}[Proof of Lemma~\ref{LMMA:Derivative}]
We first prove the differentiability with respect to $\omega$ (and thus $c$). Let $p_j(\omega_0+\tilde\omega,\sigma_0)$ and $p_j(\omega_0,\sigma_0)\equiv p_j^0$ be the solution to the wave equation with coefficients $(\omega_0+\tilde\omega,\sigma_0)$ and $(\omega_0,\sigma_0)$ respectively. Define $\hat p_j=p_j(\omega_0+\tilde\omega,\sigma_0)-p_j(\omega_0,\sigma_0)$ and $\hat{\hat p}_j=p_j(\omega_0+\tilde\omega,\sigma_0)-p_j(\omega_0,\sigma_0)-\tilde p_j^\omega$, where $\tilde p_j^\omega$ is the solution to~\eqref{EQ:Acous Sol Pert c} with $-\frac{2\tilde c}{c_0^3}$ replaced by $\tilde\omega$. We then verify that $\hat p_j$ solves
\begin{equation}\label{EQ:hp}
	\begin{array}{rcll}
  	(\omega_0+\tilde\omega)\dfrac{\partial^2 \hat p_j}{\partial t^2} -\Delta \hat p_j &=& -\tilde\omega\dfrac{\partial^2 p_j^0}{\partial t^2}, & \mbox{in}\ (0,T)\times \Omega\\
	\hat p_j(0,\bx) &=& 0, &\mbox{in}\ \Omega\\
	\dfrac{\partial \hat p_j}{\partial t}(0,\bx) &=& 0, &\mbox{in}\ \Omega\\
	\hat p_j(t,\bx) &=& 0, & \mbox{on}\ (0,T)\times\partial\Omega
	\end{array}
\end{equation}
and $\hat{\hat p}_j$ solves
\begin{equation}\label{EQ:hhp}
	\begin{array}{rcll}
  	\omega_0\dfrac{\partial^2 \hat{\hat p}_j}{\partial t^2} -\Delta \hat{\hat p}_j &=& -\tilde\omega\dfrac{\partial^2 \hat p_j}{\partial t^2}, & \mbox{in}\ (0,T)\times \Omega\\
	\hat{\hat p}_j(0,\bx) &=& 0, &\mbox{in}\ \Omega\\
	\dfrac{\partial \hat{\hat p}_j}{\partial t}(0,\bx) &=& 0, &\mbox{in}\ \Omega\\
	\hat{\hat p}_j(t,\bx) &=& 0, & \mbox{on}\ (0,T)\times\partial\Omega.
	\end{array}
\end{equation}
With the assumptions in the lemma, we conclude from standard elliptic theory~\cite{Evans-Book98,GiTr-Book00} that the diffusion equation~\eqref{EQ:Diff Backgr} admits a unique solution $u_j^0 \in \cC^4(\Omega)$. This implies that $\Gamma\sigma_0 u_j^0 \in\cC^4(\Omega)$. Therefore the wave equation~\eqref{EQ:Acous Backgr} admits a unique solution $p_j^0\equiv p_j(\omega_0,\sigma_0) \in\cC^4((0,T)\times\Omega)$ following theory in~\cite{DiDoNaPaSi-SIAM02,Isakov-Book02}. Moreover, $\frac{\partial^2 p_j^0}{\partial t^2}\in \cC^2((0,T)\times\Omega)$ and we have from~\eqref{EQ:hp} that $\hat p_j\in\cC^3((0,T)\times\Omega)$ and
\begin{multline}\label{EQ:hp d}
\|\hat p_j\|_{\cH^3((0,T)\times\Omega)}\le C_1\|\tilde\omega\dfrac{\partial^2 p_j^0}{\partial t^2}\|_{\cH^2((0,T)\times\Omega)}\le C_1\|\dfrac{\partial^2 p_j^0}{\partial t^2}\|_{\cH^2((0,T)\times\Omega)}\|\tilde\omega\|_{\cH^4(\Omega)} \\ \le C_1\|p_j^0\|_{\cH^4((0,T)\times\Omega)}\|\tilde\omega\|_{\cH^4(\Omega)} \le C_2 \|\Gamma\sigma_0 u_j^0\|_{\cH^4(\Omega)} \|\tilde\omega\|_{\cH^4(\Omega)}.
\end{multline}
Similarly, we have from~\eqref{EQ:hhp} that
\begin{multline}\label{EQ:hhp d}
\|\hat{\hat p}_j\|_{\cH^2((0,T)\times\Omega)}\le \tilde C_1\|\tilde\omega\dfrac{\partial^2 \hat p_j}{\partial t^2}\|_{\cH^1((0,T)\times\Omega)}\le \tilde C_1\|\dfrac{\partial^2 \hat p_j}{\partial t^2}\|_{\cH^1((0,T)\times\Omega)}\|\tilde\omega\|_{\cH^4(\Omega)} \\ 
\le \tilde C_1\|\hat p_j\|_{\cH^3((0,T)\times\Omega)}\|\tilde\omega\|_{\cH^4(\Omega)}.
\end{multline}
We now combine the trace theorem, ~\eqref{EQ:hp d} and~\eqref{EQ:hhp d} to get
\begin{equation}
\|\frac{\partial\hat{\hat p}_j}{\partial n}\|_{\cH^{1/2}((0,T)\times\partial\Omega)}\le \tilde C_2 \|\Gamma\sigma_0 u_j^0\|_{\cH^4(\Omega)} \|\tilde\omega\|^2_{\cH^4(\Omega)}.
\end{equation}
This shows the differentiability with respect to $\omega$.

To prove its differentiability with respect to $\sigma$, we observe that $\Lambda(c,\sigma;g_j)=\Lambda(c,H_j(\sigma);g_j)$ is linear with respect to $H_j$. Therefore, $\Lambda$ is differentiable with respect to $H_j\in\cC^4(\Omega)$ with the derivative at $(c_0, H_j^0)$ in the direction $\tilde H_j$ given as $\Lambda_{H_j}(c_0,H_j^0)\tilde H_j=\frac{\partial p_j^H}{\partial n}|_{(0,T)\times\partial\Omega}$ where $p_j^H$ solves
\begin{equation}\label{EQ:pH}
	\begin{array}{rcll}
  	\dfrac{1}{c_0^2}\dfrac{\partial^2 p_j^H}{\partial t^2} -\Delta p_j^H &=& 0, & \mbox{in}\ (0,T)\times \Omega\\
	p_j^H(0,\bx) &=& \tilde H_j, &\mbox{in}\ \Omega\\
	\dfrac{\partial p_j^H}{\partial t}(0,\bx) &=& 0, &\mbox{in}\ \Omega\\
	p_j^H(t,\bx) &=& 0, & \mbox{on}\ (0,T)\times\partial\Omega.
	\end{array}
\end{equation}
We now recall that $H_j(\sigma):\cC^4(\Omega) \to \cC^4(\Omega)$ is Fr\'echet differentiable with the derivative at $\sigma_0$ in direction $\tilde\sigma$ given as $H_{j\sigma}(\sigma_0)\tilde\sigma=\Gamma(\tilde\sigma u_j^0+\sigma_0 \tilde u_j)$ where $\tilde u_j$ solves~\eqref{EQ:Diff Sol Pert}. The chain rule of differentiation then concludes that $\Lambda$ is differentiable with respect to $\sigma$ at $\sigma_0$ and the derivative is $\Lambda_\sigma^j(c_0,\sigma_0)\tilde\sigma=\Lambda_{H_j}(c_0,\sigma_0)H_{j\sigma}(\sigma_0)\tilde\sigma$.

\end{proof}


\end{document}